\input louC.sty

\documentclass[a4paper,draft]{amsproc}
\usepackage{amsrefs}
\usepackage{amsmath}
\usepackage{mathrsfs}
\usepackage{amssymb}
\usepackage{amscd} 

\theoremstyle{plain}
 \newtheorem{thm}{\textbf{Theorem}}[section]
 \newtheorem{prop}{\textbf{Proposition}}[section]
 \newtheorem{lem}{\textbf{Lemma}}[section]
 
\theoremstyle{definition}

\theoremstyle{remark}
 
 \numberwithin{equation}{section}

\renewcommand{\leq}{\leqslant}
\renewcommand{\geq}{\geqslant}

\title[Hilbert Spaces Contractively Contained in Weighted Bergman Spaces]{Hilbert Spaces Contractively Contained in Weighted Bergman Spaces on the Unit Disk}


\author[Chu]{\bfseries Cheng Chu}

\address{
Department of Mathematics \\ 
Vanderbilt University  \\ 
Nashville, Tennessee \\
USA}
\email{cheng.chu@vanderbilt.edu}

\begin{document}

\vspace{18mm}
\setcounter{page}{1}
\thispagestyle{empty}

\begin{abstract}
Sub-Bergman Hilbert spaces are analogues of de Branges-Rovnyak spaces in the Bergman space setting. They are reproducing kernel Hilbert spaces contractively contained in the Bergman space of the unit disk. K. Zhu analyzed sub-Bergman Hilbert spaces associated with finite Blaschke products, and proved that they are norm equivalent to the Hardy space. Later S. Sultanic found a different proof of Zhu's result, which works in weighted Bergman space settings as well. In this paper, we give a new approach to this problem and obtain a stronger result. Our method relies on the theory of reproducing kernel Hilbert spaces.
\end{abstract}

\maketitle

\section{Introduction}  

Let $\DD$ denote the unit disk.
Let $L^2$ denote the Lebesgue space of square integrable functions on the unit circle $\TT$. The Hardy space $H^2$ is the subspace of analytic functions on $\DD$ whose Taylor coefficients are square summable. Then it can also be identified with the subspace of $L^2$ of functions whose negative Fourier coefficients vanish. 
The space of bounded analytic functions on the unit disk is denoted by $H^\infty$.
The Toeplitz operator on the Hardy space $H^2$ with symbol $f$ in $L^\infty(\TT)$ is defined by
$$T_f (h) = P(fh),$$ for $h\in H^2$. Here $P$ is the orthogonal projection from $L^2$ to $H^2$.

Let $A$ be a bounded operator on a Hilbert space $H$. We define the range space $\cM_H(A)=AH$, and endow it with the inner product
$$\langle Af, Ag \rangle_{\mathcal{M}_H(A)}=\langle f, g \rangle_{H},\qq f,g\in H \ominus \m{Ker}A.$$ $\cM_H(A)$ has a Hilbert space structure that makes $A$ a coisometry on $H$.
Let $b$ be a function in $H^\infty_1$, the closed unit ball of $H^\infty$. The de Branges-Rovnyak space $\cH(b)$ is defined to be the space 
$$\cM_{H^2}((I-T_b T_{\bar b})^{1/2})$$
We also define the space $\cH(\bar b)$ in the same way as $\cH(b)$, but with the roles of $b$ and $\bar{b}$ interchanged, i.e. $$\cH(\bar{b})=\cM_{H^2}((I-T_{\bar b} T_b)^{1/2}).$$ 
The spaces $\cH(b)$ and $\cH(\bar b)$ are also called sub-Hardy Hilbert spaces (the terminology comes from the title of Sarason's book \cite{sar94}).

The space $\cH(b)$ was introduced by de Branges and Rovnyak \cite{deb-rov1}. Sarason and several others made essential contributions to the theory \cite{sar94}. A recent two-volume monograph \cite{fri16-1}, \cite{fri16-2} presents most of the main developments in this area.

In this paper, we study analogues of sub-Hardy Hilbert spaces in a general setting. The Bergman space $A^2$ is the space of analytic functions on $\DD$ that are square-integrable with respect to the normalized Lebesgue area measure $dA$. For $u\in L^\infty(\DD)$, the Bergman Toeplitz
operator $\tiT_u$ with symbol $u$ is the operator on $L^2_a$ defined by 
\beq\label{berg}\tiT_u h = \tiP(uh).\eeq
Here $\tiP$ is the orthogonal projection from $L^2(\DD, dA)$ onto $A^2$.
In \cite{zhu96}, Zhu introduced the sub-Bergman Hilbert spaces. They are defined by
$$\cA(b)=\cM_{A^2}((I-\tiT_b \tiT_{\bar b})^{1/2})$$ and $$\cA(\bar{b})=\cM_{A^2}((I-\tiT_{\bar b} \tiT_b)^{1/2}).$$ Here $b$ is a function in $H^\infty_1$. It is easy to see from the definition that the spaces $\cA(b)$ and $\cA(\bar{b})$ are contractively contained in $A^2$, i.e. $\cA(b)\subset A^2$ and the inclusion map has norm at most $1$. But in most cases they are not closed subspaces of $A^2$ (see \cite{zhu96}*{Corollary 3.13}).

Sub-Bergman Hilbert spaces share some common properties with sub-Hardy Hilbert spaces as the way those spaces are defined follows from a general theory on Hilbert space contractions \cite{sar94}. For instance, both $\cH(b)$ and $\cA(b)$ are invariant under the corresponding Toeplitz operators with a co-analytic symbol \cite{sar94}*{II-7}.
One significant difference between the spaces $\cA(b)$ and $\cH(b)$ is the multipliers. The theory of $\cH(b)$ spaces is pervaded by a fundamental dichotomy, whether $b$ is an extreme point of the unit ball of $H^\infty(\DD)$. The multiplier structure of de Branges-Rovnyak spaces has been studied extensively by Lotto and Sarason in both the extreme and the nonextreme cases \cite{lot90}, \cite{lotsar91}, \cite{lotsar93}. However, Zhu \cite{zhu96} showed that every function in $H^\infty$ is a multiplier of $\cA(b)$ and $\cA(\bar b)$. As a consequence, the two sub-Bergman Hilbert spaces $\cA(b)$ and $\cA(\bar{b})$ are norm equivalent (see e.g. \cite{chu18}).

We are interested in the structure of sub-Bergman Hilbert spaces. For de Branges-Rovnyak spaces $\cH(b)$, there are two special cases: if $||b||_\infty<1$, then $\cH(b)$ is just a renormed version $H^2$; if $b$ is an inner function (i.e. $|b|=1$, a.e. on $\TT$) , then $\cH(b)$ is a closed subspace of $H^2$, called the model space. In the Bergman space setting, if $||b||_\infty<1$, then $\cA(b)$ is norm equivalent to $A^2$. But it is not known what is the space $\cA(b)$ when $b$ is a general inner function. In \cite{zhu03}, Zhu describes the space $\cA(b)$ when $b$ is a finite Blaschke product. 
\begin{thm}\label{A}
If $b$ is a finite Blaschke product, then $\cA(b)=H^2$.
\end{thm}

In \cite{sul06}, Sultanic reproved Theorem \ref{A} using a different technique, and extend the result to the weighted Bergman space setting. The main purpose of this paper is to give a new approach to this problem. We use the theory of reproducing kernel Hilbert spaces (presented in Section \ref{Pre}) and obtain stronger version of Theorem \ref{A}. 
\begin{thm}\label{m1}
For every non-constant function $b$ in $H^\infty_1$, $\cA(b)$ always contain $H^2$. Moreover, $$\cA(b)=H^2$$ if and only if $b$ is a finite Blaschke product.
\end{thm}
In Section \ref{pf}, we shall prove Theorem \ref{m1} in the weighted Bergman space setting.

\section{Reproducing kernel Hilbert spaces}\label{Pre}
In this section, we present some basic theory of reproducing kernel Hilbert spaces. For more information about reproducing kernels and their associated Hilbert spaces, see \cite{aro50}, \cite{pau16}.

Let $X\subset \CC^d$. We say a function $K: X\times X\to\CC$ is a positive kernel on $X$ (written as $K	\succeq 0$) if it is self-adjoint ($K(x,y)=\ol{K(y,x)}$), and for all finite sets $\{\Gl_1,\Gl_2,\dots, \Gl_{n}\}\subset X$, the matrix $(K(\Gl_i,\Gl_j))_{i,j=1}^n$ is positive semi-definite, i.e., for all complex numbers $\Ga_1, \Ga_2, \dots, \Ga_n$,
$$
\sum_{i,j=1}^n \Ga_i\bar{\Ga_j}K(\Gl_i, \Gl_j)\geq 0.
$$

A reproducing kernel Hilbert space $\mathcal{H}$ on $X$ is a Hilbert space of complex valued functions on $X$ such that every point evaluation is a continuous linear functional. Thus for every $w\in X$, there exists an element $K_w\in\mathcal{H}$ such that for each $f\in\mathcal{H}$, $$\langle f, K_w\rangle_{\mathcal{H}} =f(w).$$
Since $K_w(z)=\langle K_w, K_z\rangle_{\mathcal{H}}$, $K$ can be regarded as a function on $X\times X$ and we write $K(z,w)= K_w(z)$. Such $K$ is a positive kernel and the Hilbert space $\mathcal{H}$ with reproducing kernel $K$ is denoted by $\mathcal{H}(K)$.

The following theorem, due to Moore, shows that there is a one-to-one correspondence between reproducing kernel Hilbert spaces and positive kernels (see e.g. \cite{ampi}*{Theorem 2.23}).
\begin{thm}\label{Moore}
Let $X\subset \CC^d$ and let $K: X\times X\to\CC$ be a positive kernel. Then there exists a unique reproducing kernel Hilbert space $\mathcal{H}(K)$ whose reproducing kernel is $K$.
\end{thm}

For two positive kernels $K_1, K_2$, we write $$K_1	\preceq K_2$$ to mean that $$\label{n} K_2-K_1\succeq 0.$$
It is easy to check the sum of two positive kernels is still a positive kernel. The following result shows that the same holds for a product of two positive kernels, which generalizes the Schur product in matrix algebra (see e.g. \cite{aro50}).
\begin{prop}
Let $K_1, K_2$ be positive kernels on $X$. Then
$$K_1\cdot K_2\succeq 0.$$
\end{prop}
In particular, if $$K_1\preceq K_2,$$ then we can multiply another positive kernel $K$ on both sides to get $$K_1\cdot K\preceq K_2\cdot K.$$

The following proposition will be used in the proof of the main results.
\begin{prop}\label{prop}
Let $\mathcal{H}(K)$ be a reproducing kernel Hilbert space on $X$ and let $A$ be a bounded linear operator on $\cH(K)$. 
\begin{enumerate}
\item If $$||A||\leq C,$$ then 
$$
\la AK_w, AK_z\ra_{\cH(K)}\preceq C^2\la K_w, K_z\ra_{\cH(K)}.
$$
\item If $$||Af||_{\cH(K)} \geq c||f||_{\cH(K)},\q \m{for every}\q f\in \cH(K),$$ then 
$$
c^2\la K_w, K_z\ra_{\cH(K)}\preceq \la AK_w, AK_z\ra_{\cH(K)}.
$$
\end{enumerate} 
\end{prop}

\begin{proof}
We only prove (1) and (2) can be done in a similar way. 
For any distinct points $w_1,w_2,\dots, w_{n}$ in $X$ and complex numbers $\Ga_1, \Ga_2, \dots, \Ga_n$, let $$f(z)=\sum_{j=1}^{n}\Ga_jK_{w_j}(z).$$
We then have
\begin{align*}
&\sum_{i,j=1}^{n} \Ga_i\bar{\Ga_j} C^2 \la K_{w_i}, K_{w_j}\ra_{\cH(K)}-\sum_{i,j=1}^{n} \Ga_i\bar{\Ga_j} \la AK_{w_i}, AK_{w_j}\ra_{\cH(K)}\\
=&C^2 \la f, f\ra_{\cH(K)}-\la Af, Af\ra_{\cH(K)}\\
=&C^2 ||f||^2_{\cH(K)}-||Af||^2_{\cH(K)}\geq 0.
\end{align*}
It then follows from the definition of a positive kernel.
\end{proof}

We will need the next theorem that characterizes the functions that belong to a reproducing kernel Hilbert space in terms of the reproducing kernel.
\begin{thm}\cite{pau16}*{Theorem 3.11}\label{be}
Let $\mathcal{H}(K)$ be a reproducing kernel Hilbert space on $X$ and let $f: X\to\CC$ be a function. Then $f\in \mathcal{H}(K)$ with $||f||_{\mathcal{H}(K)}\leq c$ if and only if
$$
\overline{f(w)}f(z)\preceq c^2K(z,w),
$$
\end{thm}

A function $\varphi: X\to\CC$ is called a multiplier of $\mathcal{H}(K)$ on $X$ if $\varphi f\in \mathcal{H}(K)$ whenever $f\in \mathcal{H}(K)$. If $\varphi$ is a multiplier of $\mathcal{H}(K)$, let $M_\varphi: f\longmapsto \varphi f$ be the multiplication operator on $\mathcal{H}(K)$. In this case, it is well-known that the kernel functions are eigenvectors for the adjoints of multiplication operators: \beq\label{toe}
M_{\varphi}^{*}K_z=\overline{\varphi(z)}K_z.
\eeq
The following theorem characterizes multipliers of reproducing kernel Hilbert spaces (see e.g. \cite{ampi}*{Corollary 2.37}). 

\begin{thm}\label{m}
Let $\mathcal{H}(K)$ be a reproducing kernel Hilbert space on $X$, and let $\varphi: X\to\CC$ be a function. Then $\varphi$ is a multiplier of $\mathcal{H}(K)$ with multiplier norm at most $\Gd$ if and only if
$$
(\Gd^2-\varphi(z)\overline{\varphi(w)})\cdot K(z,w)\succeq 0.
$$
If $\Gd\leq 1$, then $\varphi$ is called a contractive multiplier of $\mathcal{H}(K)$.
\end{thm}

As a corollary of Theorem \ref{m}, we have
\begin{thm}\label{in}
Let $\mathcal{H}(K_1)$ and $\mathcal{H}(K_2)$ be reproducing kernel Hilbert spaces on $X$. Then $$\mathcal{H}(K_1) \subset \mathcal{H}(K_2)$$ if and only if there is some constant $\Gd>0$ such that
$$
K_1\preceq \Gd K_2.
$$
\end{thm}

The spaces we are concerned with in this paper are all reproducing kernel Hilbert spaces. One can calculate the reproducing kernel through an orthonormal basis of a reproducing kernel Hilbert space (see e.g. \cite{pau16}*{Theorem 2.4}). 
\begin{lem}\label{orn}
Let $\mathcal{H}(K)$ be a reproducing kernel Hilbert space on $X$. If $\{e_s | s\in \Gamma\}$ is an orthonormal basis for $\cH(K)$, then
$$
K(z,w)=\sum_{s\in \Gamma} \ol{e_s(w)}e_s(z).
$$
\end{lem}

The Hardy space $H^2$ has reproducing kernel 
$$
k_w^S(z)=\frac{1}{1-\bw z},
$$
which is called the Szeg\H{o} kernel. 
The Bergman space $A^2$ has reproducing kernel 
$$
k_w^0 (z)=\frac{1}{(1-\bw z)^2}.
$$

For a function $b\in H^\infty_1$, it is well-known that $b$ is a multiplier of both $H^2$ and $A^2$ and has multiplier norm equals to $1$. It follows from Theorem \ref{m} that
\beq\label{hb}
k_w^b(z)=\frac{1-\ol{b(w)}b(z)}{1-\bw z}
\eeq
and 
\beq\label{ab}
k_w^{0,b}(z)=\frac{1-\ol{b(w)}b(z)}{(1-\bw z)^{2}}
\eeq
are positive kernels. In fact, \eqref{hb} is the reproducing kernel for de Brange-Rovnyak space $\cH(b)$ \cite{sar94}*{II-3} and \eqref{ab} is the reproducing kernel for sub-Bergman Hilbert space $\cA(b)$ \cite{zhu96}*{Proposition 3.1}.

\section{Main Results}\label{pf}

To state the main results, we define more general weighted sub-Bergman Hilbert spaces. For $\Ga\geq -1$, let $A^2_\Ga$ be the reproducing kernel Hilbert space on $\DD$ with reproducing kernel
\beq\label{ker}
k^\Ga_w(z)=\frac{1}{(1-\bw z)^{\Ga+2}}.
\eeq
If $\Ga=-1$, it is just the Hardy space $H^2$. If $\Ga>-1$, $A^2_\Ga$ is the space of all analytic functions on $\DD$ that are square-integrable with respect to the measure $dA_\Ga$, where 
$$
dA_\Ga(z)=(\Ga+1)(1-|z|^2)^\Ga dA(z).
$$
In the latter case, the space $A_\Ga^2$ is called a weighted Bergman space (see e.g. \cite{zhu2}*{Chapter 4} for details).

For $\Gvp\in L^\infty(\DD)$, the weighted Toeplitz operator on $A^2_\Ga$ ($\Ga>-1$) is defined by
$$
T^\Ga_\Gvp f=P_\Ga (\Gvp f),
$$
where $P_\Ga$ is the orthogonal projection from $L^2(\DD, dA_\Ga)$ to $A^2_\Ga$.
When $\Ga=0$, $A_\Ga^2=A^2$, the standard Bergman space, and $T^\Ga_\Gvp=\tiT_\Gvp$, the Bergman Toeplitz operator defined in \ref{berg}.
Let $b\in H^\infty_1$. We define a weighted sub-Bergman Hilbert space as
$$
\cA_\Ga(b)=\cM_{A^2_\Ga}((I-T^\Ga_b T^{\Ga}_{\bar b})^{1\over 2}).
$$
The reproducing kernel of $\cA_\Ga(b)$ is
\beq\label{ker2}
k^{b, \Ga}_w(z)= \frac{1-\ol{b(w)}b(z)}{(1-\bw z)^{\Ga+2}}.
\eeq

It is easy to generalize Theorem \ref{A} to the weighted case (see e.g. \cite{sul06}*{Theorem 3.5}). 
We strengthen Theorem \ref{A} in the following two theorems. 

\begin{thm}\label{sub}
Let $b$ be a non-constant function in $H^\infty_1$. For every $\Ga\geq 0$, we have
$$A^2_{\Ga-1}\subset\cA_\Ga(b).$$
\end{thm}
\begin{proof}
By \eqref{ker}, \eqref{ker2} and Theorem \ref{in}, it is sufficient to show that
$$
\frac{1}{(1-\bw z)^{\Ga+1}} \preceq C\cdot \frac{1-\ol{b(w)}b(z)}{(1-\bw z)^{\Ga+2}}
$$
for some constant $C$.
Since 
$$
\frac{1}{(1-\bw z)^{\Ga}}\succeq 0,
$$
we only need to show
\beq\label{1}
\frac{1}{1-\bw z}	\preceq C\cdot \frac{1-\ol{b(w)}b(z)}{(1-\bw z)^{2}}.
\eeq
Let $f_0$ be the normalized reproducing kernel of $\cH(b)$ at $0$, i.e.
$$
f_0(z)=\frac{1-\ol{b(0)}b(z)}{\sqrt{1-|b(0)|^2}}.
$$
Then $f_0\in \cH(b)\cap H^\infty$ and $||f_0||_{\cH(b)}=1$.  By Theorem \ref{be},
$$
\ol{f_0(w)}f_0(z)	\preceq \frac{1-\ol{b(w)}b(z)}{1-\bw z}.
$$
Multiplying both sides by the Szeg\H{o} kernel $$\frac{1}{1-\bw z}=\la k^S_w, k^S_z\ra_{H^2},$$
we have
$$
\ol{f_0(w)}f_0(z)\la k^S_w, k^S_z\ra_{H^2} 	\preceq \frac{1-\ol{b(w)}b(z)}{(1-\bw z)^2}.
$$
By \eqref{toe}, $$T_{\barf_0} k^S_w=\ol{f_0(w)} k^S_w,$$ for every $w\in \DD$. 
We get
\beq\label{2}
\la T_{\barf_0}k^S_w, T_{\barf_0}k^S_z\ra_{H^2} 	\preceq \frac{1-\ol{b(w)}b(z)}{(1-\bw z)^2}.
\eeq

Notice that 
$f_0$ is invertible in $H^\infty$ and $$||f_0^{-1}||_\infty\leq \sqrt{\frac{1+|b(0)|}{1-|b(0)|}}.$$
Hence
$$
||h||_{H^2}\leq ||T_{\bar{f}_0^{-1}}||\cdot||T_{\bar{f}_0}h||_{H^2}=||f_0^{-1}||_{\infty}||T_{\bar{f}_0}h||_{H^2}\leq  \sqrt{\frac{1+|b(0)|}{1-|b(0)|}}||T_{\bar{f}_0}h||_{H^2}.
$$

So we have from Proposition \ref{prop} that,
$$
 C \la k^S_w, k^S_z\ra_{H^2}	\preceq\la T_{\barf_0}k^S_w, T_{\barf_0}k^S_z\ra_{H^2}.
$$
This together with \eqref{2} implies \eqref{1}.

\end{proof}

\begin{thm}\label{sub2}
Let $b$ be a non-constant function in $H^\infty_1$. For every $\Ga\geq 0$, $$A^2_{\Ga-1}=\cA_\Ga(b)$$
if and only if $b$ is a finite Blaschke product.
\end{thm}
\begin{proof}
According to Theorem \ref{sub}, $$A^2_{\Ga-1}=\cA_\Ga(b)$$ if and only if $$\cA_\Ga(b)\subset A^2_{\Ga-1}.$$

Suppose $b$ is a finite Blaschke product of degree $N$.
Then $\cH(b)$ is a Hilbert space of finite dimension $N$ (see e.g. \cite{garros}), and there exists a constant $C$ such that
$$
\frac{1-|b(w)|^2}{1-|w|^2}\leq C,
$$
for all $w\in\DD$ (\cite{zhu03}*{Lemma 1}). 
Let $\{f_n\}_{n=0}^{N-1}$ be an orthonormal basis for $\cH(b)$. Then by Lemma \ref{orn}, we have 
$$
\frac{1-\ol{b(w)}b(z)}{1-\bw z}=\sum_{n=0}^{N-1}\ol{f_n(w)}f_n(z).
$$
In particular,
$$
\sum_{n=0}^{N-1}|f_n(w)|^2=\frac{1-|b(w)|^2}{1-|w|^2}\leq C,
$$
for every $w\in\DD$. Thus for every $n$,
$$
||T_{\barf_n}||=||f_n||_\infty\leq \sqrt{C}.
$$
By Proposition \ref{prop}, 
$$
\la T_{\barf_n}k^S_w, T_{\barf_n}k^S_z\ra_{H^2}\preceq C\la k^S_w, k^S_z\ra_{H^2}.
$$
Thus
\begin{align*}
\frac{1-\ol{b(w)}b(z)}{(1-\bw z)^2}&=\sum_{n=0}^{N-1}\ol{f_n(w)}f_n(z) \la k^S_w, k^S_z\ra_{H^2}=\sum_{n=0}^{N-1}\la T_{\barf_n}k^S_w, T_{\barf_n}k^S_z\ra_{H^2}\\
&\preceq NC \la k^S_w, k^S_z\ra_{H^2}=NC\cdot \frac{1}{1-\bw z}.
\end{align*}
Multiplying both sides by the positive kernel $$\frac{1}{(1-\bw z)^{\Ga}},$$ we have 
$$
\frac{1-\ol{b(w)}b(z)}{(1-\bw z)^{\Ga+2}}\preceq \frac{NC}{(1-\bw z)^{\Ga+1}}.
$$
By Theorem \ref{in}, we get $$\cA_\Ga(b)\subset A^2_{\Ga-1}.$$

For the other direction, assume $\cA_\Ga(b)\subset A^2_{\Ga-1}$. Using Theorem \ref{in} again, we have 
$$
\frac{1-\ol{b(w)}b(z)}{(1-\bw z)^{\Ga+2}}\preceq \frac{C}{(1-\bw z)^{\Ga+1}}.
$$
for some constant $C$. 
Then
$$
\frac{1-|b(z)|^2}{(1-|z|^2)^{\Ga+2}}\leq C \frac{1}{(1-|z|^2)^{\Ga+1}},
$$
which implies
$$
\frac{1-|b(z)|^2}{1-|z|^2}\leq C,
$$
for all $z\in\DD$. Therefore, $b$ is a finite Blaschke product (see e.g. \cite{zhu03}*{Lemma 1}).
\end{proof}

{\bf Proof of Theorem \ref{m1}.} This is just a special case of Theorem \ref{sub} and \ref{sub2} when $\Ga=0$.
\qed

\bibliography{references}

\begin{bibdiv}
\begin{biblist}

\bib{ampi}{book}{
      author={Agler, J.},
      author={M\raise.45ex\hbox{c}Carthy, J.E.},
       title={{P}ick interpolation and {H}ilbert function spaces},
   publisher={American Mathematical Society},
     address={Providence},
        date={2002},
}

\bib{aro50}{article}{
      author={Aronszajn, N.},
       title={Theory of reproducing kernels},
        date={1950},
     journal={Trans. Amer. Math. Soc.},
      volume={68},
       pages={337\ndash 404},
}

\bib{chu18}{article}{
      author={Chu, C.},
       title={{Density of polynomials in sub-Bergman Hilbert spaces}},
        date={2018},
     journal={J. Math. Anal. Appl.},
      volume={467},
      number={1},
       pages={699\ndash 703},
}

\bib{deb-rov1}{book}{
      author={de~Branges, L.},
      author={Rovnyak, J.},
       title={Square summable power series},
   publisher={Holt, Rinehart, and Winston},
     address={New York},
        date={1966},
}

\bib{fri16-1}{book}{
      author={Fricain, E.},
      author={Mashreghi, J.},
       title={The theory of {$\mathcal{H}(b)$} spaces},
      series={New Mathematical Monographs, 20},
   publisher={Cambridge University Press},
     address={Cambridge},
        date={2016},
      volume={1},
}

\bib{fri16-2}{book}{
      author={Fricain, E.},
      author={Mashreghi, J.},
       title={The theory of {$\mathcal{H}(b)$} spaces},
      series={New Mathematical Monographs, 21},
   publisher={Cambridge University Press},
     address={Cambridge},
        date={2016},
      volume={2},
}

\bib{garros}{book}{
      author={Garcia, S.},
      author={Ross, W.},
       title={Model spaces: a survey},
      series={Contemp. Math.},
   publisher={Amer. Math. Soc.},
     address={Providence, RI},
        date={2015},
      volume={638},
}

\bib{lot90}{article}{
      author={Lotto, B.A.},
       title={Inner multipliers of de {Branges's} spaces},
        date={1990},
     journal={Integr. Equ. Oper. Th.},
      volume={13},
       pages={216\ndash 230},
}

\bib{lotsar91}{article}{
      author={Lotto, B.A.},
      author={Sarason, D.},
       title={Multiplicative structure of de {Branges's} spaces},
        date={1991},
     journal={Revista Matematica Iberoamericana},
      volume={7},
       pages={183\ndash 220},
}

\bib{lotsar93}{article}{
      author={Lotto, B.A.},
      author={Sarason, D.},
       title={Multipliers of de {B}ranges-{R}ovnyak spaces},
        date={1993},
     journal={Indiana Univ. Math. J.},
      volume={42},
      number={3},
       pages={907\ndash 920},
}

\bib{pau16}{book}{
      author={Paulsen, V.I.},
      author={Raghupathi, M.},
       title={An introduction to the theory of reproducing kernel {Hilbert}
  spaces},
   publisher={Cambridge University Press},
     address={Cambridge},
        date={2016},
}

\bib{sar94}{book}{
      author={Sarason, D.},
       title={{Sub-Hardy Hilbert spaces in the unit disk}},
   publisher={University of Arkansas Lecture Notes, Wiley},
     address={New York},
        date={1994},
}

\bib{sul06}{article}{
      author={Sultanic, S.},
       title={{Sub-Bergman Hilbert spaces}},
        date={2006},
     journal={J. Math. Anal. Appl.},
      volume={324},
      number={1},
       pages={639–\ndash 649},
}

\bib{zhu96}{article}{
      author={Zhu, K.},
       title={Sub-{B}ergman {H}ilbert spaces on the unit disk},
        date={1996},
     journal={Indiana Univ. Math. J.},
      volume={45},
      number={1},
       pages={165\ndash –176},
}

\bib{zhu03}{article}{
      author={Zhu, K.},
       title={Sub-{B}ergman {H}ilbert spaces on the unit disk. {II}},
        date={2003},
     journal={J. Funct. Anal.},
      volume={202},
      number={2},
       pages={327\ndash –341},
}

\bib{zhu2}{book}{
      author={Zhu, K.},
       title={Operator theory in function spaces, second edition},
      series={Mathematical Surveys and Monographs},
   publisher={American Mathematical Society},
     address={Providence, RI},
        date={2007},
      volume={138},
}

\end{biblist}
\end{bibdiv}
\end{document}